\documentclass[amssymb,twoside,12pt,amscd,reqno]{amsart}

\usepackage{latexsym,amsfonts}
\usepackage{amsmath,amsthm}          
\usepackage{euscript}                
\usepackage{epsfig}
\usepackage[all]{xy}

\theoremstyle{plain} \numberwithin{equation}{section}
\newtheorem{thm}{Theorem}[section]
\newtheorem{cor}[thm]{Corollary}

\newtheorem{conj}[thm]{Conjecture}
\newtheorem{lemma}[thm]{Lemma}

\theoremstyle{definition}
\newtheorem{remark}{Remark}[section]

\newcommand{\bi}{\begin{itemize}}
\newcommand{\ei}{\end{itemize}}
\newcommand{\bp}{\begin{proof}}
\newcommand{\ep}{\end{proof}}

\def\Tors{\text{Tors }}

\def\CC{\mathbb{C}}

\def\RR{\mathbb{R}}

\def\ZZ{\mathbb{Z}}

\begin{document}

\title{On compact complex surfaces of K\"ahler rank one}

\author[Chiose]{Ionu\c{t} Chiose$^\ast$}
\thanks{$\ast$ Supported by a Marie Curie International Reintegration
Grant within the $7^{\rm th}$ European Community Framework Programme
and the CNCSIS grant ID$\_$1185}
\author[Toma]{Matei Toma}
\thanks{}
\date{11/10/10}

\maketitle
\begin{abstract}
The K\"ahler rank of compact complex surfaces was introduced by Harvey and
Lawson in their 1983 paper on K\"ahler manifolds as a measure of ``k\"alerianity''. Here we give a partial classification of compact complex surfaces of K\"ahler rank 1. These are either elliptic surfaces, or Hopf surfaces, or they admit a holomorphic foliation of a very
 special type. As a consequence we give an affirmative answer to the question raised by Harvey and
Lawson whether the K\"ahler rank is a birational invariant.
\end{abstract}
\section*{Introduction}

In \cite{blaine}, Harvey and Lawson introduced the K\"ahler rank of a compact complex surface, a quantity intended to
measure how far a surface is from being K\"ahler. A surface has K\"ahler rank $2$ iff it is K\"ahler. 
It has K\"ahler rank $1$ iff it is not K\"ahler but still admits a closed (semi-) positive $(1,1)$-form whose zero-locus is contained in a curve. In the remaining cases, it has K\"ahler rank $0$.

Harvey and Lawson computed the K\"ahler rank of elliptic surfaces, Hopf surfaces and Inoue surfaces and conjectured that the K\"ahler rank is a birational invariant. 
In this paper we prove this conjecture. One possible approach to it is local by studying some plurisubharmonic functions on the blow-up, and showing that they are pull-backs of smooth functions. However, this approach leads to a rather involved system of differential equations, and the computations are not trivial.

Instead, we use a global approach. 
 Namely,
we will study the geometry of compact complex surfaces of K\"ahler rank $1$. The existence of a smooth positive $(1,1)$ form imposes strong restrictions on the surface. In the process we also obtain a (partial) classification of non-elliptic surfaces of K\"ahler rank $1$: they are either birational to a certain class of Hopf surfaces, or else they support a very special holomorphic foliation (in which case we conjecture that they are birational to Inoue surfaces). Recall that the K\"ahler rank of elliptic surfaces is always $1$,  \cite{blaine}.

One important technical tool that we use is the compactification of hyper-concave ends proved by Marinescu and Dinh in \cite{marinesc}.

\section{Preliminary facts}

Let $X$ be a non-K\"ahler compact complex surface, 
$P^{\infty}_{bdy}$ the set of exact positive $(1,1)$-forms on $X$ and 
$$B(X)=\{x\in X: \ \ \exists \omega\in P^{\infty}_{bdy} \ {\rm with} \ \omega(x)\neq 0\}.$$ 
Then the K\"ahler rank of $X$ is defined to be $1$ iff the complement of $B(X)$ is contained in a complex curve.

Recall that a closed positive $(1,1)$-form $\omega$ on a non-K\"ahler surface is automatically exact, cf. \cite{bpv}. Such a form is also of rank $1$, i. e., $\omega\wedge\omega=0$, hence it defines a foliation by complex curves on the set where it does not vanish. It was remarked in \cite{blaine} that the foliations induced by two such forms are the same on the set where they are both defined. One thus gets a {\it canonical foliation} on $B(X)$. In fact this foliation is induced by some closed positive $(1,1)$-form on $X$:

\begin{remark}
 There exists an exact positive $(1,1)$-form $\omega$ on $X$ such that $B(X)=\{x\in X: \ \   \omega(x)\neq 0\}$. 
\end{remark}

Indeed, we may find a sequence $(\omega_n)_n$ of exact positive $(1,1)$-forms such that for each $x\in X$ there exists some $n$ with $\omega_n(x)\neq 0$. After rescaling we get $\| \omega_n\|_{C^n}\le 2^{-n}$ and $\omega:= \sum_n\omega_n$ is the desired form.

Since all elliptic non-K\"ahler surfaces are of K\"ahler rank $1$, we shall restrict our attention to non-elliptic non-K\"ahler surfaces. By Kodaira's classification the only non-elliptic non-K\"ahler surfaces are those of class $VII$, i.e. with $b_1=1$ and
with no meromorphic functions.
The $GSS$ conjecture predicts that the minimal model of a class $VII$ surface must fall into one of the following subclasses: Hopf surfaces, Inoue surfaces and Kato surfaces. Whereas all Inoue surfaces have K\"ahler rank $1$ and all Kato surfaces have K\"ahler rank $0$, Hopf surfaces may be of K\"ahler rank $0$ or $1$ according to their type, cf. \cite{blaine}, \cite{toma}. 

\begin{remark}
 The birational invariance of the K\"ahler rank holds for the above subclasses.
\end{remark}
 It holds indeed for those surfaces whose minimal model has  K\"ahler rank $1$. 
It also holds for Kato surfaces, since these contain a cycle of rational curves and it is was shown in  \cite{toma} that in this case the K\"ahler rank is $0$. 
If $X$ is a surface whose minimal model is a Hopf surface $Y$ of K\"ahler rank $0$, then $Y$ admits an unramified finite covering $Y'$ which is a primary Hopf surface of K\"ahler rank $0$ too. It was shown in \cite{blaine} that $Y'$ admits only one exact positive current up to a multiplicative constant and this is the current of integration along its elliptic curve. But if $X$ admitted a closed positive $(1,1)$-form, its push-forward from $X$ to $Y$, and then the pull-back to $Y'$ would give a positive, non-zero, exact current which is not a multiple of an  integration current on $Y'$ (its singular locus is a finite union of points). So the  K\"ahler rank of $X$ has to vanish as well.

Let now $X$ be a non-elliptic surface of class $VII$ and
suppose the K\"ahler rank of $X$ to be $1$.

Let $\omega$ be a closed, non-zero positive $(1,1)$ form on $X$.  
Since $\omega$ is exact and $H^{1,1}_{\RR}(X)$ is naturally contained in $H^2(X,\RR)$, cf. \cite{bpv}, we can choose
a $\bar\partial$-closed $(0,1)$-form $\gamma^{0,1}$ such that 
$$\omega=\partial\gamma^{0,1}.$$
Denote by $\gamma^{1,0}$ the complex conjugate of $\gamma^{0,1}$.
Since $\omega$ is of rank $1$, it follows that $$i\gamma^{1,0}\wedge\gamma^{0,1}$$
is a positive, $i\partial\bar\partial$-closed $(1,1)$ form. The following integral is thus non-negative
$$I:=\int_Xi\gamma^{1,0}\wedge\gamma^{0,1}\wedge\omega.$$ 
We distinguish two cases, $I>0$ and $I=0$.

\section{The case $I>0$} 
\begin{thm}
Let $X$ be a class $VII$ surface and $\omega=\partial\gamma^{0,1}$ a closed, positive $(1,1)$ form
on $X$ such that $$\int_Xi\gamma^{1,0}\wedge\gamma^{0,1}\wedge\omega>0.$$ Then $X$ is birational to a Hopf surface of K\"ahler rank $1$.
\end{thm}

\begin{proof}
Let $Y$ be the minimal model of $X$. If $X$ contains a cycle of rational curves, then the K\"ahler rank of $X$ is $0$ (cf. \cite{toma}).

We will show that $X$ contains an elliptic curve. If this is the case, since it contains no cycle of rational curves, $Y$ will have to be 
 a Hopf surface (cf. \cite{nakamura}).

We assume to the contrary that $X$ has no elliptic curves and no cycles of rational curves. It is known that $X$ can carry only a finite number of complex curves, all rational. Let $(C_i)_i$ be the connected components of the maximal divisor $C$ on $X$. Consider $p:\tilde X\to X$ a regular covering of $X$ with deck transformation group isomorphic to ${\mathbb Z}$. Since there is no cycle of rational curves on $X$, each connected component of $p^{-1}(C_i)$ is isomorphic to $C_i$ via $p$. Let $\pi:X\to X'$ be the normal surface obtained from $X$ by contracting each $C_i$ to a point, and $\pi':\tilde X\to\tilde X'$ the surface obtained by contracting each connected component of $p^{-1}(C_i)$ to a point.
We then have a covering map $p':\tilde X'\to X'$ such that $p'\circ\pi'=\pi\circ p$.

Let $\omega=\partial\gamma^{0,1}$ be a positive $(1,1)$ form on $X$ such that $I=\int_Xi\gamma^{1,0}\wedge\gamma^{0,1}\wedge\omega>0$. Set $\eta=\omega+i\gamma^{1,0}\wedge\gamma^{0,1}$. 
Then $\eta$ is $i\partial\bar\partial$-closed, $\int_X\eta\wedge\eta=2\int_Xi\gamma^{1,0}\wedge\gamma^{0,1}\wedge\omega>0$ and
$\int_X\eta\wedge g>0$ where $g$ is the K\"ahler form of some  Gauduchon metric on $X$. By Buchdahl's Nakai-Moishezon criterion 
(see the Remark at the end of the paper \cite{buchdahl}), there exists a real effective divisor $D$ on $X$ and
a real ${\mathcal C}^{\infty}$ function $\psi$ on $X$ such that  
\begin{equation}\label{buchdahl}
\eta+i\partial\bar\partial\psi-if_D>0
\end{equation}
where $if_D$ is the curvature of the line bundle induced by $D$.

It is known (cf.  \cite{toma}) that there exists $\varphi$ a ${\mathcal C}^{\infty}$ function on $\tilde X$ and a representation $\rho:\Gamma\to({\mathbb R}, +)$ of the deck transformation group $\Gamma\simeq{\mathbb Z}$ of $p$
such that $p^*\omega=i\partial\bar\partial\varphi$ and $g^*\varphi=\varphi+\rho(g)$, for all $g\in\Gamma$. Then we may take $\gamma^{0,1}=i\bar\partial\varphi$.

Now let $U\subset V$ be open neighborhoods of $C$ (the maximal divisor
on X) such that $p^{-1}(V) =\bigcup_nV_n$ is a disjoint union of copies of $V$ and
$U$ is relatively compact in $V$. We can choose the Hermitian metric on
${\mathcal O}_X(D)$ such that $Supp(if_D)\subset U$. We can also assume that $\psi  > 0$ on $X$.

It is easy to check that the function  $(1+a^2\psi)e^{a\varphi}$ is strictly plurisubharmonic on $p^{-1}(X\setminus U)$ for $0<a<<1$. 

If $s$ is a section in ${\mathcal O}_X(D)$ whose zero set is $D$, the Lelong-Poincar\'e equation reads
\begin{equation*}
i\partial\bar\partial\ln\vert\vert s\vert\vert^2=[D]-if_D
\end{equation*}
Denote by $\ln\vert\vert s_n\vert\vert^2$ the function on $\tilde X$ which is $\ln\vert\vert s\vert\vert^2\circ p$ on $V_n$ and $0$ on $\tilde X\setminus V_n$.

Consider $\Phi$ a ${\mathcal C}^{\infty}$ function on $\tilde X\setminus p^{-1}(C)$ of the form
\begin{equation*}
\Phi=\varphi+(1+a^2\psi)e^{a\varphi}+\sum_na_n\ln\vert\vert s_n\vert\vert^2
\end{equation*}
where $a_n>0$. Since the function $(1+a^2\psi)e^{a\phi}$ is strictly plurisubharmonic on $p^{-1}(X\setminus U)$ and multiplicative automorphic on $\tilde X$, i.e., 
\begin{equation*}
g^*[(1+a^2\psi)e^{a\varphi}]=e^{a\rho(g)}[(1+a^2\psi)e^{a\varphi}]
\end{equation*}
for $g\in\Gamma$, it follows that we can choose $a_n$ such that $\Phi$ is a strictly plurisubharmonic function on $\tilde X\setminus p^{-1}(C)$. 

Let 
$$\Phi_-=\varphi+(1+a^2\psi)e^{a\varphi}+\sum_{n<0}a_n\ln\vert\vert s_n\vert\vert^2,$$
where we take the $V_n$ with $n<0$ in the region of $\tilde X$ where $\varphi<0$.
We denote by $\Phi'_-$ the function on $\tilde X'\setminus \pi'(p^{-1}(C))$ induced by $\Phi_-$ and with respect to this function $\tilde X'\setminus \pi'(p^{-1}(C))\cap \{\Phi'_-<0\}$ is a hyper-concave end. Then the main result of \cite{marinesc} implies that the hyper-concave end of $\tilde X'\setminus \pi'(p^{-1}(C))\cap\{\Phi'_-<0\}$ can be compactified. In particular $\tilde X'$ has only finitely many singularities, and in our case the singular set of $\tilde X'$ has to be empty due to the transitivity of the action of $\Gamma$. Therefore $X'$ is a smooth surface with no compact curves.

Since $\varphi$ is plurisubharmonic, it is constant on the connected components of $p^{-1}(C)$, therefore it descends to a continuous function $\varphi'$ on $\tilde X'$ which satisfies $g^*\varphi'=\varphi'+\rho(g),\forall g\in \Gamma$. Around each point of $\pi'(p^{-1}(C))$ consider $B(1)$ a ball of radius $1$ and $\Phi'_{\varepsilon}$ a regularization of $\Phi'$, where $\Phi'$ is the function on $\tilde X'\setminus \pi'(p^{-1}(C))$ induced by $\Phi$. Then $\Phi'_{\varepsilon}$ is strictly plurisubharmonic on $B(1-\varepsilon )$. Take $f$ a ${\mathcal C}^{\infty}$ function supported on $B(\frac 12)$, equal to $1$ on $B(\frac 14)$. If we replace $\Phi'$ by $f\Phi'_{\varepsilon}+(1-f)\Phi'$ on $B(1)$ for $0<\varepsilon <<1$, then we obtain a ${\mathcal C}^{\infty}$ strictly plurisubharmonic function on $\tilde X'$, which becomes a  
hyper-concave end with the hyper-concave end given by $\{\varphi'=-\infty\}$. Let $Y'$ be the compactification of the hyper-concave end of $\tilde X'$. It is a Stein complex space with finitely many isolated singularities. 

If we apply the above smoothing procedure to $\Phi-\varphi$ instead of $\Phi$ we obtain a multiplicative automorphic strictly plurisubharmonic function on $\tilde X'$. This is a potential of a locally conformal K\"ahler metric on $X'$. We are in position to apply the argument of \cite{ornea}, Theorem 3.1 which shows that $Y'$ is obtained from $\tilde X'$ by adding only one point. 
Moreover, if $g_0$ is the generator of $\Gamma\simeq ({\mathbb Z},+)$ such that $\rho(g_0)<0$, then $g_0$ can be extended to the whole $Y'$, and it becomes a contraction. From \cite{kato1}, Lemma 13., it follows that $Y'$ can be embedded into some affine space ${\mathbb C}^m$ such that $g_0$ is the restriction of a contracting automorphism $\tilde g_0$ of ${\mathbb C}^m$. Therefore $X'$ can be embedded into the compact Hopf manifold ${\mathbb C}^m\setminus\{0\}/<\tilde g_0>$. From \cite{kato2}, Proposition 3., we know that $X'$ contains an elliptic curve, contradiction.

Therefore the minimal model of $X$ is a Hopf surface of K\"ahler rank $1$.
\end{proof}

\section{The case $I=0$}

\begin{thm} \label{i=0}
Let $X$ be a class $VII$ surface of K\"ahler rank $1$ such that for any closed, positive $(1,1)$ form $\omega=\partial\gamma^{0,1}$ 
on $X$ one has
 $$\int_Xi\gamma^{1,0}\wedge\gamma^{0,1}\wedge\omega=0.$$ 
Then  there exists a positive, multiplicatively automorphic, non-constant, pluriharmonic function $u$ on a  ${\mathbb Z}$-covering $\tilde X$ of $X$.
\end{thm}

\begin{proof}
Again we may assume that $X$ contains no elliptic curve.

As before there exist a ${\mathbb Z}$-covering $p:\tilde X\to X$, a representation of the deck transformation group
$\rho:{\mathbb Z}\to (\RR,+)$ and a smooth function $\varphi$ on $\tilde X$ such that
$g^*\varphi=\varphi+\rho(g)$ for $g\in \mathbb Z$ and $p^*\gamma^{0,1}=i\bar\partial\varphi$.

The automorphic behaviour of $\varphi$ implies the existence of a map $f:X\to\RR/\mathbb Z\cong S^1$ such that $f^*d\theta=i\gamma^{1,0}-i\gamma^{0,1}$. Moreover $f$ is surjective and one checks easily that the map $\varphi$ is proper. 
Further we may assume that $f_*:H_1(X,\ZZ)/\Tors H_1(X,\ZZ)\to H_1(S^1,\ZZ)$ is an isomorphism.

The set $K:=\{x\in X:  \ \ df(x)=0\}$ of critical points of $f$ is compact and $S^1\setminus f(K)\neq\emptyset$ by Sard's theorem. Let $s\in S^1\setminus f(K)$ and $U\subset V\subset S^1\setminus f(K)$ two small open connected neighbourhoods of $s$ such that $U$ is relatively compact in $V$.
Thus $\gamma^{0,1}(x)\neq 0$ for all $x\in f^{-1}(V)$. 
Let further $\chi$ be some smooth non-negative, function on $S^1$ supported on $U$ and such that $\chi(s)=1$. 
We would like to pull back $\chi$ to one of the connected components of $f^{-1}(V)$. We explain first how this component is chosen.

We have a commutative diagram of smooth maps
$$
 \xymatrix{\tilde X\ar[d]^{\varphi}\ar[r]^p&X\ar[d]^f\\
   \RR\ar[r]&S^1.}
$$
We denote by $V_{\RR}$, $V_{X}$, $V_{\tilde X}$ the preimages of $V$ in $\RR$, $X$, $\tilde X$. 
Let $V'_{\RR}$ be one connected component of $V_{\RR}$ and  $V'_{\tilde X}$ its preimage in $\tilde X$. It is clear that 
$V'_{\tilde X}$ disconnects $\tilde X$.

In fact any connected component  $V''_{\tilde X}$ of $V'_{\tilde X}$  continues to disconnect $\tilde X$. Indeed, suppose it did not. Since $\varphi$ is a proper submersion over $V'_{\RR}$, each connected component  $V''_{\tilde X}$ is differentiably a product $V'_{\RR}\times F''$, where $F''$ is the corresponding component of a fiber of $\varphi$ over $V'_{\RR}$. For $p\in F''$ consider now a parametrization of the segment $V'_{\RR}\times \{p\}$. Its endpoints can also be connected by some path
in $\tilde X\setminus V''_{\tilde X}$, which connected to the previous segment induces a cycle in $H_1(\tilde X, \RR)$. This cycle is non-trivial since using $d\varphi$ and the product structure on the connected components of $V_{\tilde X}$ we may construct some closed $1$-form on $\tilde X$ with non-vanishing integral on it. Its  image in $H_1(X, \RR)$ will likewise be non-trivial. But this would contradict the fact that  $f^*:H^1(S^1,\RR)\to H^1(X,\RR)$ is an isomorphism.

Next we check that the image $V''_X$ in $X$ of one of the components $V''_{\tilde X}$ does not disconnect $X$.

\begin{lemma}
There exists $V''_X$ a connected component of $f^{-1}(V)$ such that $X\setminus \bar V''_X$ is connected.
\end{lemma}
\begin{proof}
We can assume that $f$ is differentially trivial over a small neighborhood of $\bar V$. Let $(V_{X,i})_{i=1,N}$ be the
connected components of $V_X=f^{-1}(V)$ and suppose that $X\setminus \bar V_{X,i}$ is not connected $\forall i=1,N$. Since 
the boundary of $V_{X,i}$ has two components, it follows that $X\setminus \bar V_{X,i}$ has two connected components. Set $X_0=X$. We will construct by induction connected open subsets $X_i$ of $X$ with the following properties:

i) $X_i$ is a connected component of $X_{i-1}\setminus \bar V_{i}$

ii) the restriction $H^1(X,{\mathbb R})\to H^1 (X_i,{\mathbb R})$ is injective.

Suppose $X_i$ has been constructed. Then $X_i\setminus \bar V_{i+1}$ is not connected. Indeed, suppose it is connected. Then
it follows that $X\setminus \bar V_{X,i+1}$ is connected, contradiction. Therefore $X_i\setminus \bar V_{i+1}$ is not connected and it has two connected components, $T_{i1}$ and $T_{i2}$. The Mayer-Vietoris sequence for the pair $(\bar T_{i1}\cup V_{i+1}, \bar T_{i2}\cup V_{i+1})$ implies

\begin{equation*}
0\to H^0(X_i,{\mathbb R})\to H^0(\bar T_{i1}\cup V_{i+1}, {\mathbb R})\oplus H^0(\bar T_{i2}\cup V_{i+1}, {\mathbb R})\to
\end{equation*}
\begin{equation*}
\to H^0(V_{i+1},{\mathbb R})\to H^1(X_i,{\mathbb R})\to H^1(\bar T_{i1}\cup V_{i+1}, {\mathbb R})\oplus H^1(\bar T_{i2}\cup V_{i+1}, {\mathbb R})
\end{equation*}
and since all the open subsets involved are connected, it follows that 
\begin{equation*}
H^1(X_i,{\mathbb R})\to H^1(\bar T_{i1}\cup V_{i+1}, {\mathbb R})\oplus H^1(\bar T_{i2}\cup V_{i+1}, {\mathbb R})
\end{equation*}
is injective. Given that $H^1(X,{\mathbb R})\to H^1 (X_i,{\mathbb R})$ is injective, it follows that 
\begin{equation*}
H^1(X,{\mathbb R})\to H^1(\bar T_{i1}\cup V_{i+1}, {\mathbb R})\oplus H^1(\bar T_{i2}\cup V_{i+1}, {\mathbb R})
\end{equation*}
is injective.
Hence we can assume that the restriction $H^1(X,{\mathbb R})\to H^1(\bar T_{i1}\cup V_{i+1}, {\mathbb R})$ is injective. Set
$X_{i+1}=T_{i1}$. For $i=N$ we obtain $X_N$ a connected open subset of $X$ such that $X_N\cap (\cup_{i=1}^N V_{X,i})=
\emptyset$ and $H^1(X,{\mathbb R})\to H^1(X_N, {\mathbb R})$ injective. But this is impossible since $f^*d\theta$ is exact on $X\setminus f^{-1}(\bar V)$.
\end{proof}

Let $V''_X$ a component of $V_X$ chosen as above and $\chi''$ the pull-back of $\chi$ to $V''_X$ extended trivially to $X$. We adopt the same notations for the preimages of $U$. 
Consider the positive 
$(1,1)$-form 
$i\chi''\gamma^{1,0}\wedge\gamma^{0,1}$. Since $\gamma^{0,1}\wedge\omega=0$, it follows that this form is $d$-closed hence $d$-exact. Thus there exist a real constant $c$ and a smooth real function $\psi$ on $X$ such that

$$i\chi''\gamma^{1,0}\wedge\gamma^{0,1}=\partial(c\gamma^{0,1}+i\bar\partial\psi)=i\partial\bar\partial v,$$
where $v=\varphi +\psi$ has additive automorphy on $\tilde X$.
We shall show that this form does not vanish identically on $X\setminus \bar U''_X$.

Assume that $c\gamma^{0,1}+i\bar\partial\psi=0$ on $X\setminus  \bar U''_X$. Then the function $v=c\varphi+\psi$ is plurisubharmonic on $V''_X$ and holomorphic on $X\setminus\bar U''_X$. But this function is real and thus constant on $
X\setminus\bar U''_X$. By the maximum principle the plurisubharmonic function $c\varphi+\psi$ reaches its maximum on the boundary of $V''_X$ and since it is constant near this boundary it is constant everywhere on $V''_X$. This contradicts the fact that $\chi\neq 0$.

Thus $\partial v$ is a holomorphic $(1,0)$-form on $X\setminus U''_X$ and the $(1,1)$-form  $i\partial v\wedge \bar\partial v$ is positive and $d$-exact hence it induces a holomorphic foliation on $X\setminus U''_X$ which coincides with the canonical foliation on $B(X)\setminus U''_X$.

We can now do this construction for two disjoint open connected subsets $U_1$ and $U_2$ of $V$ 
and obtain additively automorphic plurisubharmonic functions $v_1$ and $v_2$ on  $\tilde X$. 
Set $W:=X\setminus(\bar U''_{1,X}\cup \bar U''_{2,X})$ and let $W_1$, $W_2$, $W_3$ be the connected components of $V''_X\setminus(\bar U''_{1,X}\cup \bar U''_{2,X})$ counted in such a way that $W_2\cup\bar U''_{1,X}\cup \bar U''_{2,X}$ is connected. We have $W_1\cup W_2 \cup
 W_3\subset W$.
If $X$ contains some rational curves then these are necessarily contained in fibers of $f$ and we have chosen $V$  such that no such curves are contained in $f^{-1}(V)$.

The  $1$-forms $\partial v_j$ are holomorphic and 
 $d$-closed on  $W_1\cup W_2 $
 hence locally of the form $dh_j$ for some holomorphic functions $h_j$. 
Notice that on $W_1\cup W_2 $ the canonical foliation is regular and induced by $\partial
\varphi$ and at the same time by $\partial
(v_j+\varphi)$ and by $\partial v_j$.
 The irreducible components of the level sets of $h_j$ are therefore contained in its leaves and further contained 
 in fibers of $f$.
 It follows that the zero set of $\partial v_j$ being complex analytic and closed, it will consist of only isolated points.
 Since the holomorphic $1$-forms $\partial v_1$ and $\partial v_2$ induce the same foliation on  $W_1\cup W_2$
they must be
  proportional so
 there exists 
a
meromorphic function $h$  on  $W_1\cup W_2$ 
such
that $ \partial v_1=h\partial v_2$. The zeroes of  $\partial v_j$
 being isolated, the function $h$ is holomorphic and without zeroes on  $W_1\cup W_2$. 
Thus $\partial h\wedge\partial v_2=0$ and the restriction of $dh$ to the leaves of the foliation vanishes, hence $h$ is constant on these leaves. But the Zariski closure of a general leaf is a connected component of  $W_1\cup W_2$ and the function $h$ has to  be constant on $W_1$ and on $W_2$.

We thus get constants $k_1$, $C_1$ such that $v_1=k_1v_2+C_1$ on $W_1$. Since $v_1$, $v_2$ are real and non-constant, one sees that $k_1$, $C_1$ are also real. We replace now the function $v_2$ by $k_1v_2+C_1$. This new function remains pluriharmonic on  $X\setminus f^{-1}(\bar U_2)$ and agrees with $v_1$ on $W_1$.  On $W_2$ we have now
$v_1=k_2v_2+C_2$ as before for some $k_2\in\RR^*$ and $C_2\in\RR$. 
Remark that $k_2\neq 1$, otherwise $\partial v_1$ and $\partial v_2$ would glue well  to give some nontrivial holomorphic $1$-form on $X$, which is absurd.

Ading now the same constant $C:=\frac{C_2}{k_2-1}$ to both $v_1$ and $v_2$ we obtain

$$v_1 +C=k_2(v_2+C)$$
on $W_2$.
We thus get by gluing a pluriharmonic function $v$ with multiplicative automorphic factor $k_2$ on $\tilde X$. If $k_2$ is not positive we may pass to a double cover of $X$ and obtain here a positive multiplicative factor $\lambda=k_2^2$.

Now we check that $v$ has no zeros on $\tilde X$. Suppose that the zero set $Z$ of $v$ is not empty. We have seen that $v$ is constant on the leaves of the canonical foliation, hence it is locally constant on the fibers of $\varphi$. Conversely $\varphi$ is locally constant on the fibers of $v$. Then $Z$ is a union of connected components of fibers of 
$\varphi$. Take a component of $Z$ contained in a domain $D$ bordered by two regular connected fibers of $\varphi$ on which $v$ does not vanish. We may assume that $v$ has positive values on these fibers, say $a$ and $\lambda a$. But then $v|_{\bar D}$ attains its minimum in the interior of $D$, which is a contradiction.

We finally set $u=\pm v$ according to the sign of 
 $v$ and get the desired positive, multiplicatively automorphic, non-constant, pluriharmonic function on $\tilde X$.
\end{proof}

\section{Corollaries and remarks}

\begin{remark}
 If $u$ is a positive, multiplicatively automorphic, non-constant, pluriharmonic function on $\tilde X$, then 
$\partial u$ gives a non-trivial $d$-closed section in $H^0(\Omega^1_X\otimes L)$, where $L$ is a flat line bundle on $X$. 
Moreover
$$i\partial\bar\partial (-\log u)=\frac{i\partial u\wedge \bar\partial u}{u^2}$$
descends to a closed positive $(1,1)$-form on $X$ which vanishes at most on some rational curves. In particular the K\"ahler rank of $X$ is $1$. 
The level sets of $u$ are compact Levi flat real hypersurfaces of $X$ saturated under the canonical foliation.
\end{remark}

Since a pluriharmonic function descends by blowing down we get the desired

\begin{cor}
 The K\"ahler rank is a bimeromorphic invariant.
\end{cor}

We also get

\begin{cor}
 The canonical foliation of a surface of K\"ahler rank $1$ is always the restriction 
to $B(X)$ of a (possibly singular) holomorphic foliation on $X$.
\end{cor}

Notice that we haven't used in our proof the assuption that $B(X)$  contained a dense Zariski open set of $X$. So we get the following 

\begin{cor}
 The K\"ahler rank of a compact complex surface $X$ is equal to the maximal rank
that a positive closed $(1,1)$-form may attain at some point of $X$.
\end{cor}

\begin{remark}
 If the minimal model of $X$ is a Hopf surface, then $X$ cannot admit  
a positive, multiplicatively automorphic, non-constant, pluriharmonic function $u$ on a  ${\mathbb Z}$-covering of $X$.
\end{remark}

Indeed, if the minimal surface of $X$ is a Hopf surface, then $X$ will have an elliptic curve whose preimage in  $\tilde X$ is isomorphic to $\CC^*$. But then the restriction of any positive, pluriharmonic function to this preimage is constant and a non-trivial multiplicatively automorphic behaviour of such a function on  $\tilde X$ is impossible. 

Recall that every non-K\"ahler compact complex surface admits  some non-trivial $d$-exact positive $(1,1)$-current, cf \cite{Lam99'}.
In \cite{toma} one defines the {\em modified K\"ahler rank} of a compact complex surface in the following way: it is $2$ if the surface admits a  K\"ahler metric, and if not $0$ or $1$ depending on whether there is one or several non-trivial $d$-exact positive $(1,1)$-currents up to multiplication by positive constants. One immediately checks that the modified K\"ahler rank is a birational invariant. The  K\"ahler rank and the modified K\"ahler rank of Kato surfaces are computed in  \cite{toma} and it shows up that they need not coincide but in this case the modified K\"ahler rank is at least as large  as the  K\"ahler rank. 

From the proof of Theorem \ref{i=0} it follows that there are infintely many $d$-exact positive $(1,1)$-forms up to multiplication by positive constants on surfaces of class $VII$ with $I=0$ as in the statement. The same is obviously true for non-K\"ahler elliptic surfaces and it holds for Hopf surfaces of K\"ahler rank $1$ by \cite{blaine}.Thm. 58. Thus we get:

\begin{cor}
 The modified K\"ahler rank of a compact complex surface is at least as large  as its  K\"ahler rank. 
\end{cor}

Let us mention in conclusion that the classification of compact complex surfaces of K\"ahler rank $1$ would be complete if we had a positive answer to the following:
\begin{conj}
A non-elliptic surface $X$ admitting a positive, multiplicatively automorphic, non-constant, pluriharmonic function on a  ${\mathbb Z}$-covering should be bimeromorphically equivalent to an Inoue surface.
\end{conj}

 \providecommand{\bysame}{\leavevmode\hbox
to3em{\hrulefill}\thinspace}

Addresses:
\par\noindent{\it Ionu\c t Chiose}: \\
Institute of Mathematics of the Romanian Academy\\
{\tt Ionut.Chiose@imar.ro}\\

\par\noindent{\it Matei Toma}: \\
Institut Elie Cartan, UMR 7502, 
Nancy-Universit\'e - CNRS - INRIA,\\
B.P. 70239, 54506 Vandoeuvre-l\`es-Nancy Cedex, France \\
{\tt toma@iecn.u-nancy.fr}

\end{document}